\documentclass{amsart}

\usepackage{amsthm,amssymb,latexsym,amsmath,color,mathrsfs}
\usepackage[all]{xy}
\usepackage[utf8]{inputenc}
\usepackage{geometry}
\usepackage{marginnote,comment}

\newcommand{\iz}{I_Z}\newcommand{\oz}{\mathcal{O}_Z}

\newcommand{\oc}{\mathcal{O}_C}
\newcommand{\mO}{\mathcal{O}}
\newcommand{\Sing}{\operatorname{Sing}}
\newcommand{\Z}{{\mathbb Z}}
\newcommand{\C}{\mathbb{C}}

\newcommand{\K}{\mathbb{K}}
\newcommand{\p}[1]{{\mathbb{P}^{#1}}}

\newcommand{\op}[1]{{\mathcal O}_{\mathbb{P}^{#1}}}

\newcommand{\ox}{{\mathcal O}_{X}}

\newcommand{\sing}{\operatorname{Sing}}

\newcommand{\sI}{\mathscr{I}}

\newcommand{\cO}{\mathcal{O}}

\newcommand{\inext}{{\mathcal E}{\it xt}}

\newcommand{\Ext}{\operatorname{Ext}}
\newcommand{\Hom}{\operatorname{Hom}}

\DeclareMathOperator{\Aut}{Aut}
\DeclareMathOperator{\coker}{coker}
\DeclareMathOperator{\im}{im}

\DeclareMathOperator{\rk}{{rk}}

\DeclareMathOperator{\Pic}{{Pic}}

\newcommand{\nf}{N_{\mathscr{F}}}

\def\sF{{\mathscr{F}}}

\newcommand{\into}{\hookrightarrow}
\newcommand{\onto}{\twoheadrightarrow}

\newtheorem{theorem}{Theorem}[section]
\newtheorem{mthm}{Main Theorem}
\newtheorem{proposition}[theorem]{Proposition}

\newtheorem{lemma}[theorem]{Lemma}
\newtheorem{corollary}[theorem]{Corollary}

\theoremstyle{definition}
\newtheorem{remark}[theorem]{Remark}
\newtheorem{example}[theorem]{Example}
\newtheorem{definition}[theorem]{{\bf Definition}}

\title{Foliations by curves on threefolds}

\author[A. Cavalcante]{Alana Cavalcante}
\address{DECEA - UFOP\\
Departamento de Ciências Exatas e Aplicadas\\
Rua 36 - Loanda, 115 \\
35931-008 João Monlevade - MG, Brazil}
\email{alana.decea@ufop.edu.br}

\author[M. Jardim]{Marcos Jardim}
\address{IMECC - UNICAMP \\ Departamento de Matem\'atica \\
Rua S\'ergio  Buarque de Holanda, 651\\ 13083-970 Campinas-SP, Brazil}
\email{jardim@ime.unicamp.br}

\author[D. Santiago]{Danilo Santiago}
\address{IMECC - UNICAMP \\ Departamento de Matem\'atica \\
Rua S\'ergio  Buarque de Holanda, 651\\ 13083-970 Campinas-SP, Brazil}
\email{d192309@dac.unicamp.br}

\begin{document}

\maketitle

\begin{abstract}
We study the conormal sheaves and singular schemes of 1-dimensional foliations on smooth projective varieties $X$ of dimension 3 and Picard rank 1. We prove that if the singular scheme has dimension 0, then the conormal sheaf is $\mu$-stable whenever the tangent bundle $TX$ is stable, and apply this fact to the characterization of certain irreducible components of the moduli space of rank 2 reflexive sheaves on $\mathbb{P}^3$ and on a smooth quadric hypersurface $Q_3\subset\mathbb{P}^4$. 
Finally, we give a classification of local complete intersection foliations, that is, foliations with locally free conormal sheaves, of degree 0 and 1 on $Q_3$.
\end{abstract}

\tableofcontents

\section{Introduction}\label{sec1}
	
The study of holomorphic foliations on complex manifolds is a classical topic of research that goes back to the end of the $19^{\rm th}$ century, though the qualitative study of polynomial differential equations by Poincaré, Darboux, and Painlevé, and currently with ramifications to complex geometry and algebraic geometry. We are presently interested in the latter, and our objective is to apply algebraic geometric techniques to understand the normal and conormal sheaves and singular schemes of foliations of dimension 1 on smooth projective varieties of dimension 3. 

So let $X$ be a smooth projective threefold $X$ of Picard rank 1. Let $\ox(1)$ denote the ample generator of $\Pic(X)$, and given a sheaf $F$ on $X$ we set $F(r):=F\otimes\ox(1)^{\otimes r}$, as usual. Let $TX$ denote the tangent bundle of $X$ and define
$$ \tau_X := \min\{t \in \Z ~|~ H^0(TX(t)) \neq 0\} ~~,~~
\rho_X:= \min\{t \in \Z ~|~ H^0(\Omega^1_X(t)) \neq 0 \} ~~, $$
$$ \nu_X := \int_X H^3 ~~~~{\rm and}~~~~ c_X := \int_X c_1(\Omega^1_X)\cdot H^2 = - \int_X c_1(TX)\cdot H^2 , $$
where $H:=c_1(\ox(1))$ is the ample generator of $\Pic(X)$.

A \emph{foliation by curves $\sF$ on $X$} is a short exact sequence of the form
\begin{equation}\label{folcurve}
	\sF ~~:~~ 0\to\ox(-r-\tau_X) \stackrel{\sigma}{\to} TX \to \nf \to 0
\end{equation}
where $\nf$ is a torsion free sheaf called the \emph{normal sheaf} of $\sF$; this is also known in the literature as a \emph{Pfaff field of rank 1}. The non negative integer $r$ above is called the \emph{degree} of $\sF$; we note that this definition is in general different from the one proposed in \cite[Definition 2]{CJ} for Pfaff fields, though both definition coincide for $X=\p3$. Note that $\rk(\nf)=2$.

The image of the morphism $\sigma^\vee:TX\to\ox(\tau_X+r)$ is the twisted ideal sheaf $\sI_Z(r+\tau_X)$ of a subscheme of $X$ of dimension at most 1, called the \emph{singular scheme} of $\sF$ and denoted by $\sing(\sF)$. Thus dualizing the sequence in display \eqref{folcurve} we obtain
\begin{equation}\label{dual}
	0 \to \nf^\vee \to \Omega^1_X \stackrel{\sigma^\vee}{\to} \sI_Z(r+\tau_X) \to 0,
\end{equation}
where $\nf^\vee$ is called the \emph{conormal sheaf of $\sF$}.

In general, the singular scheme $Z:=\sing(\sF)$ may contain a pure 1-dimensional subscheme obtained as follows. Let $U$ be the maximal 0-dimensional subsheaf of $\oz$; the quotient sheaf $\oz/U$ must be the structure sheaf of pure dimension 1 scheme, call it $C$. We therefore obtain the exact sequence
\begin{equation} \label{sequence I}
	0\to U \to \oz \to \oc \to 0
\end{equation}
or, equivalently,
\begin{equation} \label{sequence II}
	0 \to \sI_Z \to \sI_C \to U \to 0.
\end{equation}
The scheme $C$ is called the \emph{1-dimensional component} of the singular scheme of $\sF$, and it is denoted by $\sing_1(\sF)$.

Foliations by curves on threefolds have not been widely considered so far. A systematic study of for the case $X=\p3$ was initiated in \cite{NC} and continued in \cite{CJM}. Furthermore, the authors of \cite{AMS} consider foliations by curves on Fano threefolds, obtaining results regarding the connectedness of the singular scheme $\sing(\sF)$ and the stability of the conormal sheaf $\nf^\vee$. In all of these papers the focus was on foliations whose singular scheme has pure dimension 1, which implies that conormal sheaf must be locally free; we will consider arbitrary foliations by curves and a wider class of threefolds, generalizing many of the results obtained in \cite{AMS,CJM,NC}. We emphasize that this is a considerable step forward for two reasons: first, rank 2 locally free sheaves are much more restrictive class in comparison to rank 2 reflexive sheaves; second, a generic vector field $\sigma\in H^0(TX(r+\tau_X)$ gives rise to a foliation by curves as in display \eqref{folcurve} only isolated singularities, and foliations by curves with locally free conormal sheaves only occur in high codimension. 

The first goal of this paper is to provide a relation between the Chern classes of the conormal sheaf $\nf^\vee$ and the discrete invariants of the singular scheme $Z$, namely the length of $U$ and the degree and genus of $C$, see Section \ref{sec:sing} below. In particular, we show that
\begin{equation} \label{int-c3 intro}
	\int_X c_3(TX(r+\tau_X)) =
	- \int_X c_3(\Omega^1_X(-r-\tau_X)) =
	h^0(U) + \sum_j \mu_X(C_j,r),
\end{equation}
where $\sing_1(\sF)=\bigsqcup_j C_j$ is the partition of the curve $\sing_1(\sF)$ into its connected components, and 
\begin{equation} \label{mu-inv intro}
	\mu_X(C,r) := 3(r+\tau_X)\deg(C) + 2\chi(\oc) .
\end{equation}

Next, we turn our attention to describing properties of the conormal sheaf of a foliation by curves. We study its Betti numbers $h^p(\nf^\vee(k))$ and give criteria that guarantee its stability (in the sense of Mumford--Takemoto). 

Recall that a torsion free sheaf $F$ is said to be $\mu$-\emph{(semi)stable} (with respect to $H$) if every subsheaf $L\subset F$ for which $F/L$ is torsion free satisfies
$$ \dfrac{c_1(L)\cdot H^2}{\rk(L)} < (\le)~ \dfrac{c_1(F)\cdot H^2}{\rk(F)}. $$
In particular, note that if $TX$ is $\mu$-stable, then $\tau_X>c_X/3\nu_X$. 

\begin{mthm} \label{mthm-stable}
	Let $\sF$ be a foliation by curves of degree $r$ satisfying $\dim\sing(\sF)=0$ on a smooth projective threefold $X$ with $\Pic(X)=\Z$ such that $h^1(\ox(t))=0$ for all $t \in \Z$. 
	\begin{enumerate}
		\item Assume that $\dim Z=0$; if $r > (\geq) ~ c_X/\nu_X - 3\tau_X,$ then $N_{\sF}$ is $\mu$-(semi)stable; if $TX$ is $\mu$-stable, then $N_{\sF}$ is $\mu$-stable for every $r$.
		\item If $r <(\leq) ~ 2 \rho_X - \tau_X  + c_X/\nu_X,$ then the conormal sheaf of a foliation by curves of degree $r$ on $X$ is $\mu$-(semi)stable.
	\end{enumerate}
\end{mthm}

This is proved in Section \ref{sec:stable}. We observe that the hypothesis $h^1(\ox(t))=0$ for all $t \in \Z$ is satisfied by every Fano threefold, while $TX$ is $\mu$-stable whenever $X$ is a smooth weighted projective complete intersection Fano threefold with Picard number equal to one.

The set of vector fields $\sigma\in H^0(TX(r+\tau_X))$ for which $\dim\coker\sigma^\vee=0$ is an open subset of $\mathbb{P}( H^0(TX(r+\tau_X))$. For this reason, foliations by curves satisfying $\dim\sing(\sF)=0$ are called \emph{generic}. Therefore, the first part of Main Theorem \ref{mthm-stable} implies that generic foliations by curves of degree $r$ provide a family of $\mu$-stable rank 2 reflexive with given Chern classes parametrized by and open subset of $\mathbb{P}( H^0(TX(r+\tau_X))$. It turns out that such families are dense within an irreducible component of the (Gieseker--Maruyama) moduli space of stable rank 2 sheaves on the projective space $\p3$ and
while only defined a closed subset within an irreducible component of the moduli space of stable rank 2 sheaves on a smooth quadric threefold $Q_3$. The following theorem, proved in Sections \ref{sec:p3} and \ref{sec:q3}, arises as an interesting application of Main Theorem \ref{mthm-stable} to the study of moduli spaces of rank 2 reflexive sheaves on threefolds.

\begin{mthm} \label{mthm-p3q3}
	\begin{enumerate}
		\item The moduli space of stable rank 2 sheaves on $\p3$ with Chern classes
		$$ (c_1,c_2,c_3) = \left\{ \begin{array}{l}
			(0,3k^2+4k+2,8k^3+16k^2+12k+4),~~ k\ge1 \\ (-1,3k^2+k+1,8k^3+4k^2+2k+1),~~ k\ge0
		\end{array}\right. $$
		contains a rational irreducible component whose generic point is the normal sheaf of a generic foliation be curves on $\p3$.
		\item The moduli space of stable rank 2 sheaves on $Q_3$ with Chern classes
		$$ (c_1,c_2,c_3) = \left\{ \begin{array}{l}
			(0,(3k^2+6k+4)H^2,(8k^3+24k^2+26k+6)H^3),~~ k\ge1 \\ (-H,(3k^2+3k+1)H^2,(8k^3+12k^2+8k-2)H^3),~~ k\ge0
		\end{array}\right. $$
		possesses an irreducible component which contains, as a closed subset, the normal sheaves of generic foliations by curves on $Q_3$.
	\end{enumerate}
\end{mthm}

It is worth remarking that Main Theorem \ref{mthm-stable} and \ref{mthm-p3q3} are parallel to \cite[Theorem 1]{CCJ2} and \cite[Theorem 3]{CCJ2}, respectively, concerning generic \emph{codimension one distributions} on threefolds.

Finally, we consider \emph{local complete intersection (LCI) foliations}, which are defined as foliations by curves with locally free conormal sheaves; the nomenclature is motivated by the fact that they are given locally as the intersection of two codimension one distributions. When the conormal sheaf $\nf^\vee$ splits as a sum of line bundles, we say that $\sF$ is a \emph{complete intersection (CI) foliation}; CI foliations by curves on Fano thereefolds are studied in \cite{AMS}, where characterizations in terms of the singular scheme are provided. Here, motivated by the classification of LCI foliations by curves of low degree on $\p3$ given in \cite{CJM}, we give the first steps towards a classification of LCI foliations by curves on smooth quadric hypersurfaces in $\p4$ of degree 0 and 1. To be precise, we prove in Section \ref{sec:lci q3}:

\begin{mthm} \label{mthm-lci}
	Let $\sF$ be a local complete intersection foliation of degree d a smooth quadric hypersurface $Q_3\subset\p4$.
	\begin{enumerate}
		\item For $d=0$, we have that $\nf^\vee=S(-1)$, where $S$ is the spinor bundle, and $\sing(\sF)$ is the disjoint union of a line and a conic.
		\item For $d=1$, we have that $E:=\nf^\vee(2)$ and $C:=\sing(\sF)$ are one of the following three possibilities:
		\begin{enumerate}
			\item[(2.1)] $E$ is the $\mu$-stable bundle with Chern classes $c_1(E)=0$ and $c_2(E)=2L$ and $C$ is a rational curve of degree 6;
			\item[(2.2)] $E$ is the $\mu$-semistable bundle with Chern classes $c_1(E)=0$ and $c_2(E)=2L$ and $C$ is a curve of degree 6 given by the union of a rational and an elliptic curves.
			\item[(2.3)] $E=\mathcal{O}_Q^{\oplus 2}$ and $C$ is a connected curve of degree 8 and genus 3.
		\end{enumerate}
	\end{enumerate}
\end{mthm}


\section{Preliminary results}\label{prelims}

Following the notation set up in the Introduction, we will use this section to go over some details that will be useful later on. 

We begin by observing that the dualization of the sequence in display \eqref{folcurve} also yields:
\begin{equation}\label{ext-n}
	\inext^1(\nf,\ox)\simeq \oz(r+\tau_X) ~~{\rm and}~~
	\inext^p(\nf,\ox)=0 ~~{\rm for}~~ p=2,3.
\end{equation}
In particular, note that if $\sF$ is not generic, then $\nf$ is not reflexive, since in this case $\inext^1(\nf,\ox)$ would be a 0-dimensional sheaf.  

In addition, we remark that the conormal sheaf $\nf^\vee$, which is always reflexive, is locally free if and only if $\sF$ has no isolated singularities, that is $U=0$ and $\sing(\sF)=\sing_1(\sF)$. Indeed, dualizing the exact sequences in displays \eqref{dual} and \eqref{sequence II}, we obtain
\begin{equation} \label{exts}
	\inext^1(\nf^\vee,\ox) \simeq \inext^2(\iz(r+\tau_X),\ox) \simeq 
	\inext^3(U,\ox);
\end{equation}
in particular, note that 
\begin{equation} \label{c3h0}
	h^0(\inext^1(\nf^{\vee},\ox)) = h^0(\inext^3(U,\ox)) = h^0(U), 
\end{equation}
where the last inequality follows from Serre duality.

Given a rank 2 reflexive sheaf $F$ on a projective threefold $X$ with Picard rank 1, one can show, using the same argument as in \cite[Proposition 2.6]{RH2}, that
\begin{equation} \label{c3h0'}
	\int_X c_3(F) = h^0(\inext^1(F,\ox)); 
\end{equation}
indeed, first note that if $Q$ is a 0-dimensional sheaf on $X$, then Grothendieck--Riemann--Roch implies that $2h^0(Q)=\int_Xc_3(Q)$.
Recall that $F$ admits a resolution $0\to L_1\to L_0\to F\to 0$, where $L_1$ and $L_0$ are locally free sheaves; dualizing this sequence, one obtains, since $F^\vee\simeq F\otimes\det(F)^\vee$
$$ 0 \to F \to L_0^\vee\otimes\det(F) \to L_1^\vee\otimes\det(F) \to \inext^1(F,\ox) \to 0. $$
Comparing Chern classes, one concludes that $2c_3(F)=c_3(\inext^1(F,\ox))$, and integration leads to the equality in display \eqref{c3h0'}.

Letting $\sF$ be a foliation by curves, we gather the equalities in displays \eqref{c3h0} and \eqref{c3h0'} to obtain
\begin{equation} \label{c3h0''}
	\int_X c_3(\nf^\vee) = h^0(U).
\end{equation}

Finally, the simplest examples of foliations by curves are the complete intersection ones, given by sequences of the form 
$$ 0\to \ox(-r_1-\rho_X)  \oplus \ox(-r_2-\rho_X) \to 
\Omega^1_X \to \sI_C(r+\tau_X) \to 0 , $$
where $r=c_X/\nu_X+r_1+r_2+2\rho_X-\tau_X$.
Local complete intersection foliations by curves on hypersurfaces $X\subset\p4$ can be constructed as follows.

\begin{example}
	Let $F$ be a globally generated rank $2$ locally free sheaf on a threefold hypersurface $X\subset\p4$. Note that $\Omega^1_{X}(2)$ is globally generated, since we have epimorphisms
	$$ \Omega^1_{\p4}(2) \twoheadrightarrow \left( \Omega^1_{\p4}(2)\right)|_X \twoheadrightarrow \Omega^1_{X}(2) $$
	and $\Omega^1_{\p4}(2)$ is globally generated. We then have that $F\otimes\Omega^1_{X}(2)$ is also globally generated, thus \cite[Teorema 2.8]{O} implies that there exists an injective morphism $\phi:F^\vee\to\Omega^1_X(2)$ that degenerates in codimension at least 2. It follows that $\coker\phi$ is a torsion free sheaf of rank 1, therefore
	$$ \sF ~~:~~ 0 \rightarrow F^{\vee}(-2)  \stackrel{\phi}{\rightarrow} \Omega^1_{X} \rightarrow \sI_Z (r+\tau_X) \rightarrow 0 $$
	is a LCI foliation by curves on $X$ with
	$r = c_1(\Omega^1_X) - c_1(F^\vee(-2)) - \tau_X = d  + c_1(F) - \tau_X - 1,$
	where $d$ is the degree of $X.$
\end{example} 



We recall that a vector bundle $F$ on $X$ is $0-$Buchsbaum if and only if $F$ has no intermediate cohomology,
i.e. $H^i(X,F(t)) = 0$ for every $t \in \mathbb{Z}$ and $1 \leq i \leq n-1.$ This bundles are also called arithmetically Cohen Macaulay.
If $F$ has all the intermediate cohomology modules with trivial structure, it is called $1-$Buchsbaum (i.e. arithmetically Buchsbaum).

We will use the following relation between the conormal sheaf of a foliation by curves and its singular locus \cite{AMS}.

\begin{theorem} \cite[Theorem 1.2]{AMS}
	Let $\sF$ be a distribution of dimension one on a smooth weighted projective complete intersection Fano threefold $X,$ with index $\iota_X$ and Picard number one. If $Z=\Sing(\sF)$ is the singular scheme of $\sF,$ then:
	\begin{enumerate}
		\item If $N^{\vee}_{\sF}$ is arithmetically Cohen Macaulay, then $Z$ is arithmetically Buchsbaum, with \\ $h^1(X, \sI_Z(r)) =1$ being the only nonzero intermediate cohomology for $H^i(\sI_Z).$
		\item If $Z$ is arithmetically Buchsbaum with $h^1(X, \sI_Z(r)) =1$ being the only nonzero intermediate cohomology for $H^i(\sI_Z)$ and $h^2(N^{\vee}_{\sF}) =h^2(N^{\vee}_{\sF}(-c_1(N^{\vee}_{\sF}) - \iota_X)) =0$ and $\iota_X \in {1, 2, 3},$ then $N^{\vee}_{\sF}$ is arithmetically Cohen Macaulay.
	\end{enumerate}
\end{theorem}


\section{Properties of the singular scheme}\label{sec:sing}

Our first goal is to provide a relation between the Chern classes of the conormal sheaf and the numerical invariants of the singular scheme of a foliation by curves on threefold $X$. Grothendieck--Riemann--Roch implies that $c_3(Q)=(2h^0(U)/\nu_X)\cdot H^3$, where $U$ is a 0-dimensional sheaf on $X$; moreover, if $C\subset X$ is a curve, then $c_2(\oc)=c_2(\sI_C)=[C]$ and
$$ \chi(\oc) =  \dfrac{1}{2} \int_X\left([C]\cdot c_1(TX) + c_3(\oc) \right) ~~\Longrightarrow~~
\int_X c_3(\oc) = 2\chi(\oc) + \int_X [C]\cdot c_1(\Omega^1_X). $$
Therefore, we obtain
\begin{equation}\label{ast11}
	c_3(\oc) = -c_3(\sI_C) = \dfrac{2}{\nu_X}\chi(\oc)\cdot H^3 + [C]\cdot c_1(\Omega^1_X) .
\end{equation}
With these facts in mind, we are finally ready to state the main result of this section.

\begin{theorem} \label{invariantes}
	Let $\sF$ be a foliation by curves of degree $r$ on threefold $X,$
	with $\Pic(X) = \Z\cdot H,$ where $H$ is the class of a hyperplane section, i.e. $H = c_1(\mathcal{O}_X(1)).$ Then,
	\begin{align*}
		c_1(\nf^\vee) = & ~ c_1(\Omega^1_X) - (r+\tau_X)H; \\
		c_2(\nf^\vee) = & ~ c_2(\Omega^1_X)- c_1(\Omega^1_X) (r+\tau_X)H + (r+\tau_X)^2 H^2 - [C]; \\
		c_3(\nf^\vee) = & ~ (h^0(U)/\nu_X)H^3 = -c_3 (\Omega^1_X(-r-\tau_X)) - 3  (r+\tau_X)[C]H - (2\chi(\cO_C)/\nu_X) H^3. \\
	\end{align*}
\end{theorem}

\begin{proof}
	Use $c(\Omega^1_X)=c(\nf^\vee)\cdot c(\sI_{Z}(r+\tau_X))$ to obtain
	\begin{align*}
		c_1(\Omega^1_X) = & ~ c_1(\nf^\vee)+c_1(\sI_{Z}(r+\tau_X)); \\
		c_2(\Omega^1_X) = & ~ c_2(\nf^\vee)+ c_1(\nf^\vee) \cdot c_1(\sI_{Z}(r+\tau_X))+c_2(\sI_{Z}(r+\tau_X)); \\
		c_3 (\Omega^1_X) = & ~ c_3(\nf^\vee) + c_3(\sI_{Z}(r+\tau_X)) + c_1(\nf^\vee)\cdot c_2(\sI_{Z}(r+\tau_X)) + \\
		& ~ c_2(\nf^\vee)\cdot c_1(\sI_{Z}(r+\tau_X)).
	\end{align*}
	
	The first equation gives $c_1(\nf^\vee)=c_1(\Omega^1_X) - (r+ \tau_X)H.$ From the exact sequence (\ref{sequence I}), it follows that $c_2(\sI_{Z}(r+\tau_X))=c_2(\sI_{C}(r+\tau_X))=[C]$, thus substitution into the second equation yields
	$$ c_2(\nf^\vee)= c_2(\Omega^1_X) - c_1(\Omega^1_X) \cdot (r+\tau_X)H + (r+\tau_X)^2H^2 -[C]. $$
	
	Moreover, the substituting the expressions for the first and second Chern classes into the third equation we obtain
	\begin{eqnarray}
		c_3(\Omega^1_X) & = & c_3(\nf^\vee)+ c_3(\sI_{Z}(r+\tau_X)) + c_2(\Omega^1_X) (r+\tau_X)H  \label{c3} \\ 
		& & -2(r+\tau_X) [C]H - c_1 (\Omega^1_X) (r+\tau_X)^2 H^2 + c_1 (\Omega^1_X)[C] \nonumber \\ & &  +(r+\tau_X)^3 H^3. \nonumber
	\end{eqnarray}
	
	Using the sequence in display \eqref{sequence II} and the formula in \eqref{ast11}, we obtain
	\begin{equation}\label{ast1}
		c_3(\sI_{Z}) = c_3(\sI_C)-c_3(U) = -\dfrac{1}{\nu_X}\left( 2\chi(\oc) + 2h^0(U)\right)\cdot H^3 - [C]\cdot c_1(\Omega^1_X)
	\end{equation}
	thus $c_3(\sI_Z(r+\tau_X))=c_3(\sI_Z)-(r+\tau_X)[C]H$. 
	Substituting  into the equation in display (\ref{c3}), we obtain
	$$ c_3(\Omega^1_X)-(r+\tau_X)c_2(\Omega^1_X)H + (r+\tau_X)^2c_1(\Omega^1_X)H^2 - (r+\tau_X)^3H^3 = $$ 
	$$ = c_3(\nf^\vee) -\dfrac{1}{\nu_X}\left( 2\chi(\oc) + 2h^0(U)\right)\cdot H^3 - 3(r+\tau_X) [C]H, $$
	and note that the left hand side of the previous equality (written in the top line) matches \linebreak $c_3(\Omega^1_X(-r-\tau_X))$. Using that $c_3(\nf^\vee) = (h^0(U)/\nu_X)H^3,$ we have that
	$$ (h^0(U)/\nu_X) H^3 = - c_3 (\Omega^1(-r-\tau_X)) - 3 (r+\tau_X)[C]H -\dfrac{2}{\nu_X}\chi(\cO_C)H^3, $$
	as claimed.
\end{proof}

In particular, we obtain the following expected result.

\begin{corollary}
	If $\sF$ is a generic foliation by curves of degree $r$ on a smooth projective threefold with $\Pic(X)=\Z$, then the length of $\sing(\sF)$ is equal to 
	$$ -\int_X c_3(\Omega^1_X(-r-\tau_X)) = 
	\int_X c_3(TX(r+\tau_X)). $$ 
\end{corollary}

We observe that, in the previous statement, the singular locus of a generic foliation by curves need not be reduced and may contain multiple points.

Next, note that the degree of a curve $C\subset X$ is defined as follows:
\begin{equation}\label{deg-curve}
	\deg(C) = \int_X [C]\cdot H.
\end{equation}
The following result is obtained simply by integrating the second and third identities in Theorem \ref{invariantes}.

\begin{corollary}\label{deg+chi}
	If $\sF$ is a local complete intersection foliation by curves of degree $r$ on a smooth projective threefold with $\Pic(X)=\Z$, then
	\begin{align*}
		\deg(C) & = \int_X \left( c_2(\Omega_X^1) -c_2(\nf^{\vee}) \right) H - (r+\tau_X)c_X + (r+\tau_X)^2\nu_X \\
		\chi(\oc) & = -\dfrac{1}{2} \int_X \left( c_3(\Omega^1(-r-\tau_X)) + 3(r+\tau_X)[C]H \right)
	\end{align*}
\end{corollary}

Given a positive integer $r>0$, we introduce the following invariant for a connected curve $C\subset X$:
\begin{equation} \label{mu-inv}
	\mu_X(C,r) := 3(r+\tau_X)\deg(C) + 2\chi(\oc) .
\end{equation}
With this notation in mind, the third equality in Theorem \ref{invariantes} can be rewritten in the following manner
\begin{equation} \label{int-c3}
	\int_X c_3(\Omega^1_X(-r-\tau_X)) =
	- \int_X c_3(TX(r+\tau_X)) =
	-h^0(U) - \sum_j \mu_X(C_j,r),
\end{equation}
where $\sing_1(\sF)=\bigsqcup_j C_j$ is the partition of the curve $\sing_1(\sF)$ into its connected components. 

In other words, for any foliation by curves of degree $r$, right hand side of the equality in display \eqref{int-c3} depends only on $X$ and $r$. The first term can be understood as the contribution of isolated singularities, counted with multiplicity, while the second term can be understood as the contribution of each connected component of $\sing_1(\sF)$.

Next, motivated by the consideration in the previous paragraphs, we show how to determine the number of connected components of $\sing_1(\sF)$ for a foliation by curves $\sF$ in terms of the conormal sheaf. Set 

\begin{theorem} \label{connected cod2}
	Let $\sF$ be a non-generic foliation by curves of degree $r$ on a smooth threefold $X$ such that $h^1(\ox)=0$. If $h^p(\Omega^1_X(-r-\tau_X)))=0$ for $p=1,2$, then
	$$ h^0(\oc) = h^2(\nf^\vee(-r-\tau_X)) + 1 - \int_X c_3(\nf^{\vee}). $$
	where $C:=\sing_1(\sF)$. In particular, if $C$ is reduced, then it is connected if and only if $h^2(\nf^\vee(-r))=\int_X c_3(\nf^\vee)$.
\end{theorem}

We observe that the hypothesis $h^1(\ox)=0$ holds for hypersurfaces in $\p4$ and Fano threefolds, while the hypothesis $h^1(\Omega^1_X(-r-\tau_X)))=0$ holds for smooth
weighted projective complete intersection Fano threefold with Picard number equal to one when $r \neq - \tau_X.$
Separating these varieties by the index and comparing the values of $-r - \tau_X$ for which $h^2(\Omega^1_X(-r-\tau_X)) =0$, we can see that the common vanishing of cohomology group, occurs when $r < -\tau_X-4.$

\begin{proof}
	Consider the exact sequence \ref{dual} defining the distribution $\sF$; twisting it by $\mathcal{O}_{X}(-r-\tau_X)$ and passing to cohomology we obtain,
	$$ H^1 (\Omega^1_X(-r-\tau_X)) \to H^1(\sI_Z) \to  H^2(\nf^\vee(-r-\tau_X)) \to H^2(\Omega^1_X(-r-\tau_X)), $$  
	thus
	$$ h^0(\oz) - 1 = h^1(\sI_Z) =  h^2(\nf^\vee(-r-\tau_X)), $$
	where the first equality follows from the standard sequence $0 \to \sI_Z \to \ox \to \oz \to 0$. Using the sequence in display \eqref{sequence I}, we obtain
	$$ h^0(\oc) = h^1(\oz) - h^0(U). $$
	Using the equality in display \eqref{c3h0'}, we obtain
	$$ h^0(\oc) = h^2(\nf^\vee(-r-\tau_X)) + 1 - \int_X c_3(\nf^{\vee}). $$
	The second statement follows from the fact that $C$ is connected if and only if $h^0(\oc)=1$, when $C$ is reduced.
\end{proof}


As an application of Theorem \ref{connected cod2}, we provide conditions that guarantee the connectedness of the $1$-dimensional component of the singular set of a foliation by curves.

\begin{corollary}\label{q-conn}
Let $\sF$ be a foliation by curves on a Fano threefold $X$ of degree $r$; assume that $C:=\sing_1(\sF)$ reduced.
\begin{enumerate}
\item If $\iota_X = 4$ and $r \neq  1$, then $C$ is connected if and only if $h^2(\nf^\vee(1-r))=\int_{\p3} c_3(\nf^\vee)$.
\item If $\iota_X = 3$ and $r \neq \{0,1\}$, then $C$ is connected if and only if $h^2(\nf^\vee(-r))=\int_{Q_3}c_3(\nf^\vee)$.
\item If $\iota_X = 2$ and $r<-4-\tau_X$, then $C$ is connected if and only if $h^2(\nf^\vee(-r-\tau_X))=\int_X c_3(\nf^\vee)$.
\item If $\iota_X = 1$ and $r<-3-\tau_X$, then $C$ is connected if and only if $h^2(\nf^\vee(-r-\tau_X))=\int_X c_3(\nf^\vee)$.
\end{enumerate}
\end{corollary}








\section{Cohomology of the conormal sheaf}\label{sec:cohomology}

This section is dedicated to the study of the cohomology ring of the conormal sheaf of a foliations by curves on a smooth weighted projective complete intersection Fano threefold, setting up some technical results that will be useful later on. We start by establishing some vanishing results.

\begin{lemma}\label{cohomQ3}
	If $\sF$ is a foliation by curves on a smooth weighted projective complete intersection Fano threefold $X,$ then
	\begin{enumerate}
		\item[(i)] $h^0(\nf^\vee (t))=0\mbox{ for } t \leq 1$;
		\item[(ii)] $h^1(\nf^\vee(t))=0$ for $t\le-r-\tau_X$; 
		\item[(iii)] $h^3(\nf^\vee(t))=0$ for $t \geq r+ \tau_X-1$.
	\end{enumerate}
\end{lemma}

\begin{proof}
	For item (i), we consider the exact sequence in display \eqref{dual} twisted by $\mathcal{O}_{Q_3}(t)$ and after we taking the long exact sequence of cohomology. By \cite[Lemma 5.17]{CM}, $h^0(\Omega_{X}^1(t))=0$ for $t \leq 1$.
	
	Item (ii) is obtained considering the following piece of the long exact cohomology sequence
	$$\cdots \to H^0(X,\sI_Z(r+\tau_X+t)) \to H^1(X,\nf^\vee(t)) \to H^1(X,\Omega_{X}^1(t)).$$
	The term on the left vanishes when $t\le-r-\tau_X,$ while the term of the right vanishes for all $t \neq 0.$
	
	For item (iii), by Serre duality we get
	$h^3(\nf^\vee(t)) = h^0 ((\nf^\vee)^{\vee}(-t-\iota_X)).$ 
	Since $\nf^\vee$ is a rank two reflexive sheaf,
	$$ (\nf^\vee)^{\vee} = \nf^\vee(-c_1(\nf^\vee)) = \nf^\vee(\iota_X+r+\tau_X). $$ Thus, by item (i), $h^0(\nf^\vee(r+\tau_X-t))=0$ for $t\geq r+\tau_X-1.$
\end{proof}

Now we dualize the sequence in display \eqref{dual} obtaining
\begin{equation}\label{dualdual}
	0 \to \ox(-r-\tau_X) \stackrel{\sigma}{\to} TX \to \nf^{\vee\vee} \stackrel{\zeta}{\to} \omega_C\otimes\omega_X^\vee(-r-\tau_X) \to 0,
\end{equation}
where $C:=\sing_1(\sF)$ and $\omega_C$ is its dualizing sheaf. Set $G:=\ker\zeta$, and consider the short exact sequences
\begin{equation}\label{break}
	0 \to \ox(-r-\tau_X) \stackrel{\sigma}{\to} TX \to G \to 0
	~~{\rm and}~~ 
	0\to G \to \nf^{\vee}\otimes\det(\nf) \stackrel{\zeta}{\to} \omega_C\otimes\omega_X^\vee(-r-\tau_X) \to 0,
\end{equation}
where we use $\nf^{\vee\vee}\simeq\nf^{\vee}\otimes\det(\nf)$ since $\nf^\vee$ is a reflexive rank 2 sheaf.

\begin{lemma}\label{cohom2}
	If $\sF$ is a foliation by curves on a smooth weighted projective complete intersection Fano threefold $X,$ then
	\begin{enumerate}
		\item[(i)] $h^2(\nf^\vee (t))=0\mbox{ for } t > -2 \iota_X+1$;
		\item[(ii)] $h^1(\nf^\vee(t))=h^1(\omega_C\otimes\omega_X^\vee(t-r-\tau_X))$ for $t> \max\{8-\iota_X-r-\tau_X, -2 \iota_X+1\}$;
	\end{enumerate}
\end{lemma}
\begin{proof}
	Using the two sequences in display \eqref{break}, we obtain
	$$h^2(G(t)) = h^0(\mO_X(-t+r+\tau_X-\iota_X)),$$
	since
	$h^2(TX(t))= h^3(TX(t)) = 0$ for $t> -\iota_X$ and $h^3(\mO_X(t-r-\tau_X)) =h^0(\mO_X(-t+r+\tau_X-\iota_X)),$ by Serre duality.
	So, $h^2(G(t)) = 0 $ for
	$t>r+\tau_X-\iota_X+1, $ and hence $h^2(\nf^{\vee}(t))=0$ for $t> -2\iota_X+1.$
	
	For item $(ii)$, separating the varieties by the index and comparing the values of $t+\iota_X$ for which $h^1(\Omega^2_X(t+\iota_X))=h^1(TX(t))=0,$ we can see that the common vanishing of cohomology group, occurs when $t>8.$ So, $h^1(G(t))=0$ for $t>8,$ and $h^1(\nf^{\vee}\otimes\det(\nf)) =h^1(\omega_C\otimes\omega_X^\vee(t-r-\tau_X))$ for $t> \mbox{max} \{8, r+\tau_X-\iota_X+1\}.$ Therefore,  $$h^1(\nf^\vee(t))=h^1(\omega_C\otimes\omega_X^\vee(t-r-\tau_X)) \;\; \mbox{for} \;\;
	t> \mbox{max}\{8-\iota_X-r-\tau_X, -2 \iota_X+1\}.$$
	
\end{proof}

Finally, we state two corollaries on generic foliations on $X=\p3$ and $X=Q_3$ that will be specifically used below.

\begin{corollary}\label{cohfolp3}
	If $\sF$ is a generic foliation by curves of degree $r$ on $\p3$, then: 
	\begin{enumerate}
		\item[(1)] $h^0(\nf(t))=0$ for $t\leq -2;$
		\item[(2)] $h^1(\nf(t))=0$ for all $t\in\mathbb{Z}$;
		\item[(3)] $h^2(\nf(t))=h^0(\op3(-t+r-5))$
		for ($t\neq -4$ and $t\geq -5$), moreover $h^2(\nf(t))=0$ for $t\geq r-4$;
		\item[(4)] $h^3(\nf(t))=0$ for $t\geq -5$.
		
	\end{enumerate}
\end{corollary}

\begin{corollary}\label{cohquadrica}
	If $\sF$ is a generic foliation by curves of degree $r$ on $Q_3$, then:
	\begin{enumerate}
		\item[1)] $h^0(\nf(t))=0$ for $t\leq -1$; 
		\item[2)] $h^1(\nf(-2))=1$ and $h^1(\nf(t))=0$ for $t\neq -2;$
		\item[3)] $h^2(\nf(t))=h^0(\mathcal{O}_{Q_3}(-t+r-3))$ for $t\geq -4$ and $t\neq -3$,  in particular, $h^2(\nf(t))=0$ for $t\geq r-2;$
		\item[4)] $h^3(\nf(t))=0$ for $t\geq -4.$
	\end{enumerate}	
\end{corollary}


\section{Stability results}\label{sec:stable}

In this section, we explore conditions under which one can guarantee that the conormal sheaf of a foliation by curves is $\mu$-semistable. We begin by looking into generic foliations by curves.


\begin{theorem} \label{estabilidade}
	Let $\sF$ be a generic foliation by curves of degree $r$ on a smooth projective threefold $X$ with $\Pic(X)=\Z$ such that $h^1(\ox(t))=0$ for all $t \in \Z$. If $r > (\geq) ~ c_X/\nu_X - 3\tau_X,$ then $N_{\sF}$ is $\mu$-(semi)stable. If, in addition, $TX$ is $\mu$-stable, then $N_{\sF}$ is $\mu$-stable for every $r$.
\end{theorem}

We observe that the hypothesis $h^1(\ox(t))=0$ for all $t \in \Z$ is satisfied by every Fano threefold, while $TX$ is $\mu$-stable whenever $X$ is a smooth weighted projective complete intersection Fano threefold with Picard number equal to one.

\begin{proof}
	Any rank 1 locally free subsheaf $\ox(-t)\into\nf$ induces a nontrivial section in $H^0(\nf(t))$. Twisting the exact sequence in display \eqref{folcurve} by $\ox(t)$ and taking cohomology we obtain a surjective map
	$$ H^0(TX(t)) \onto H^0(\nf(t)) , $$
	since $h^1(\ox(t-r-\tau_X))=0$ for all $t \in \Z$. Thus, if $h^0(\nf(t)) \neq 0,$ then $t \geq \tau_X$.
	By hypothesis,
	$$ \int_X c_1(\nf)\cdot H^2 = -c_X + (r + \tau_X)\nu_X \geq -2 \tau_X\nu_X $$
	It follows that
	$$ \int_X c_1(\ox(-t))\cdot H^2 = -t\nu_X \leq \dfrac{1}{2}\int_X c_1(N_{\sF})\cdot H^2, $$
	thus $N_{\sF}$ is $\mu$-semistable. Assuming the strict inequality $r > c_X/\nu_X - 3\tau_X$, we conclude that
	$$ \int_X c_1(\ox(-t))\cdot H^2 < \dfrac{1}{2}\int_X c_1(N_{\sF})\cdot H^2, $$
	thus $N_{\sF}$ is $\mu-$stable.
	
	When $TX$ is $\mu$-stable, then we have 
	$$ -\tau_X\nu_X < \dfrac{1}{3}\int_X c_1(TX)\cdot H^2 ~~ \Longrightarrow ~~
	c_X -3\tau_X\nu_X<0. $$
	Since $r\ge0$, the inequality in the hypothesis is automatically satisfied, and we conclude that the normal sheaf of a generic foliation by curves is automatically satisfied.
\end{proof}

Next, we consider non-generic foliations by curves, showing that if the degree is sufficiently small, then the conormal sheaf is $\mu$-(semi)stable.  

\begin{theorem}\label{stable2}
	Let $X$ be a smooth projective threefold with $\Pic(X)=\Z$. If $r <(\leq) ~ 2 \rho_X - \tau_X  + c_X/\nu_X,$ then the conormal sheaf of a foliation by curves of degree $r$ on $X$ is $\mu$-(semi)stable.
\end{theorem}

\begin{proof}
	If $\nf^\vee$ is not $\mu$-stable, then there exists a nontrivial section in $H^0(\nf^\vee(k))\ne 0$ where
	$$ -k\nu_X \ge \dfrac{1}{2} \int_X c_1(\nf^\vee)H^2 ~~\Longrightarrow~~ k\nu_X \le  
	-\dfrac{1}{2}\left( c_X - (r+\tau_X)\nu_X \right). $$
	It follows that $h^0(\Omega^1_X(k))\ne0$, thus $k\ge \rho_X$, and 
	$$ \rho_X\nu_X \le -\dfrac{1}{2}\left( c_X - (r+\tau_X)\nu_X \right) ~~\Longrightarrow~~
	r \ge 2\rho_X -\tau_X + c_X/\nu_X. $$
	For the claim about $\mu$-semistability, one must only change the inequalities by strict ones.
\end{proof}

\begin{remark}
	For $X=\p3$, we have that $(2 \rho_X - \tau_X)\nu_X - c_X=1$, so the previous result guarantees that the conormal sheaf of a foliation by curves of degree 0 is $\mu$-stable, while the conormal sheaf of a foliation by curves of degree 1 is $\mu$-semistable. However, the following complete intersection foliation by curves of degree 2 on $\p3$
	$$ 0 \to \op3(-2)\oplus\op3(-3) \to \Omega^1_{\p3} \to \sI_Z(1) \to 0 $$
	has a conormal sheaf which is not $\mu$-semistable. This example shows that the inequality in Theorem \ref{stable2} is sharp.
\end{remark}

As an application of Theorem \ref{estabilidade}, we provide an existence result for stable reflexive sheaves with given Chern classes.

\begin{corollary}
	Let $X$ be a smooth projective threefold with rank one Picard group. Then, for each integer $r > c_X/\nu_X - 3\tau_X$, there exists a $\mu$-stable rank 2 reflexive sheaf $E$ with Chern classes:
	\begin{itemize}
		\item $c_1(E) = c_1(\Omega^1_X) - (r+ \tau_X)H;$
		\item $c_2(E) = c_2(\Omega^1_X) - (r+ \tau_X)H c_1(\Omega^1_X) + (r+ \tau_X)^2 H^2;$
		\item $c_3(E) = -c_3(\Omega^1_X(-r-\tau_X)).$
	\end{itemize}
\end{corollary}




In the next two sections, we will further study these sheaves on $X=\p3$ and on $X$ being a smooth quadric hypersurface in $\p4$.

\section{Generic foliations by curves on $\p3$}\label{sec:p3}

Recall that a generic foliation by curves $\mathscr{F}$ on $X=\p3$ is given by
$$\sF ~~:~~ 0\to\op3(-r+1) \stackrel{\sigma}{\to} T\p3 \to \nf \to 0$$
since $\tau_{\mathbb{P}^3}=-1$, where $r\ge0$ is the degree of $\sF$. According to Theorem \ref{estabilidade}, the normal sheaf $\nf$ is a $\mu$-stable rank 2 reflexive sheaf on $\p3$.

When $\sF$ has odd degree, say $r=2k+1$, the normalization of the normal sheaf fits into the short exact sequence 
\begin{equation}\label{folp3}
	0\to\op3(-2-3k) \stackrel{\sigma}{\to} T\p3(-2-k) \to \nf(-2-k) \to 0,
\end{equation}
for $k\geq 0$. Similarly, if $\sF$ has even degree, say $r=2k$, then the normalization of the normal sheaf fits into the short exact sequence 
\begin{equation}\label{folp3even}
	0\to\op3(-1-3k) \stackrel{\sigma}{\to} T\p3(-2-k) \to \nf(-2-k) \to 0,
\end{equation}
where $k\geq 0$.

For generic foliations by curves of odd degree, i.e., those given by the exact sequence in display \eqref{folp3}, we show the following theorem:

\begin{theorem}\label{teofolp3}
	For each $k\geq 1$, the moduli space of stable rank 2 reflexive sheaves on $\mathbb{P}^3$ with Chern classes
	\begin{equation*}
		(c_1,c_2,c_3)=(0,3k^2+4k+2,8k^3+16k^2+12k+4)
	\end{equation*}
	\noindent contains a rational irreducible component of dimension $4k^3+20k^2+31k+14$ whose generic point is the normal sheaf of a generic foliation of degree $2k+1$ on $\p3$ given by the exact sequence in display \eqref{folp3}.
\end{theorem}

Before starting the proof this theorem, we note that the family of sheaves $\nf$ given by the exact sequence in display \eqref{folp3}, which we will denote simply by $\mathcal{G}(2k+1)$,  has dimension $h^0(T\mathbb{P}^3(2k))-1$, since each sheaf $\nf$ is defined by a section 
$$\sigma\in \Hom(\op3(-2-3k),T\mathbb{P}^3(-2-k))\simeq H^0(T\mathbb{P}^3(2k))$$
up to a scalar multiple, i.e. $\sigma\in H^0(T\mathbb{P}^3(2k))$, so we must argue that the following equality holds

\begin{equation*}
	\dim\Ext^1(\nf,\nf)= \dim\mathcal{G}(2k+1)=h^0(T\mathbb{P}^3(2k))-1=4k^3+20k^2+31k+14,
\end{equation*}
for each $k\geq 0$.  

Being $\nf$ a stable rank 2 reflexive sheaf on $\p3$ with $c_1(\nf)=0$, we have

\begin{equation*}
	\dim\Ext^1(\nf,\nf)-\dim\Ext^2(\nf,\nf)=8c_2(\nf)-3=24k^2+32k+13,
\end{equation*}
see \cite[Proposition 3.4]{RH2}.

Therefore, we must to compute the dimension of $\Ext^2(\nf,\nf)$, showing that
\begin{equation*}
	\dim\Ext^2(\nf,\nf)=h^0(T\mathbb{P}^3(2k))-24k^2-32k-14=4k^3-4k^2-k+1.
\end{equation*}

\begin{proof}[\textit{Proof of the Theorem \ref{teofolp3}}]
	
	Applying the functor $\Hom(.,\nf(-2-k))$ to the exact sequence in display
	\eqref{folp3}, 
	we get the isomorphism
	
	\begin{equation}\label{eq33}
		\Ext^2(\nf,\nf)\simeq H^2(\Omega^{1}_{\p3}\otimes \nf)   
	\end{equation}
	since $h^1(\nf(2k))=h^2(\nf(2k))=0$ by Corollary $\ref{cohfolp3}$.
	
	In order to compute $h^2(\Omega^1_\p3\otimes \nf)$,  we twist the dual Euler sequence
	
	\begin{equation*}
		\xymatrix{
			0 \ar[r] & \Omega^{1}_{\p3} \ar[r] & \op3(-1)^{\oplus 4} \ar[r] & \op3 \ar[r] & 0 \\
		}  
	\end{equation*}
	by $\otimes \nf$ and pass to cohomology, obtaining the exact sequence in cohomology
	
	\begin{equation*}
		\xymatrix{
			0 \ar[r] & H^2(\Omega^{1}_{\p3} \otimes \nf) \ar[r] & H^2(\nf(-1)^{\oplus 4}) \ar[r] & H^2(\nf) \ar[r] & 0, \\
		}  
	\end{equation*}
	since $H^1(\nf)=H^3(\Omega^1_\p3\otimes \nf)=0.$ 
	Thus, we get the equality
	
	\begin{equation*}
		h^2(\Omega^1_\p3\otimes \nf)=4.h^2(\nf(-1))-h^2(\nf).
	\end{equation*}
	
	Now, using the item $(3)$ of the Corollary $\ref{cohfolp3}$ and the isomorphism $(\ref{eq33})$, we get
	
	\begin{equation*}
		\dim \Ext^2(\nf,\nf)= 4k^3-4k^2-k+1,
	\end{equation*}
	for $k\geq 1$ and this ends the proof.
\end{proof}

Similarly, if a foliation by curves $\sF$ on $\p3$ is given by short exact sequence in display \eqref{folp3even}, i.e. has degree even, then the normal sheaf $\nf$ has Chern classes 

\begin{equation*}
	\begin{array}{c}
		c_1(\nf) = -1 ,  \\
		c_2(\nf) = 3k^2+k+1,\\
		c_3(\nf) = 8k^3+4k^2+2k+1. 
	\end{array}
\end{equation*}

Moreover, the family of this sheaves has dimension

\begin{equation*}
	\dim \mathcal{G}(2k)=h^0(T\p3(2k-1))-1=4k^3+14k^2+14k+3.
\end{equation*}

Following the proof of the Theorem $\ref{teofolp3}$, its is easy to show that

\begin{equation*}
	\dim\Ext^2(\nf,\nf)=4k^3-6k^2+6k,
\end{equation*}
for $k\geq 0$ and hence  
$$\dim \mathcal{G}(2k)=\dim\Ext^1(\nf,\nf),$$
since 
\begin{equation*}
	\dim\Ext^1(\nf,\nf)-\dim\Ext^2(\nf,\nf)=8c_2(\nf)-2c_1(\nf)^2-3=24k^2+8k+3,
\end{equation*}
for $k\geq 0$, see \cite[Proposition 3.4]{RH2}. 

As in the case $c_1=0$, we have:

\begin{theorem}
	For each $k\geq 0$, the moduli space of stable rank 2 reflexive sheaves on $\mathbb{P}^3$ with Chern classes
	\begin{equation*}
		(c_1,c_2,c_3)=(-1,3k^2+k+1,8k^3+4k^2+2k+1)
	\end{equation*}
	\noindent contains a rational, irreducible component of dimension $4k^3+14k^2+14k+3$ whose generic point is the normal sheaf of a generic foliation of degree $2k$ on $\mathbb{P}^3$ given by the exact sequence in display 
	\eqref{folp3even}.
\end{theorem}

In the next section we will do an analogous study when $X=Q_3$ is a smooth quadric hypersurface in $\p4$.


\section{Foliations by curves on quadric threefolds} \label{sec:lci q3}

Let $Q_3$ denote a smooth quadric hypersurface in $\p4.$ Let $H$ be the class of a hyperplane section, so that 
$$\Pic(Q_3) = H^2(Q_3,\Z) = \Z H ~~. $$
Moreover, the cohomology ring
$H^{*}(Q_3,\Z)$ is generated by H, a line $L \in H^4(Q_3,\Z)$ and a point $P \in H^6(Q_3,\Z)$
with the relations: $H^2 = 2L, \; H . L = P, \; H^3 = 2P.$ In addition, we note that
$$ \tau_{Q_3} = 0 ~~,~~ \rho_{Q_3}=2 ~~,~~ \nu_{Q_3}=2 ~~{\rm and}~~ c_{Q_3}=-3\nu_{Q_3}=-6. $$

In this section, we will focus on foliations by curves of low degree on smooth quadric hypersurfaces. Our first remark is a direct application of Theorem \ref{stable2}; since $2\rho_{Q_3}-\tau_{Q_3} + c_{Q_3}/\nu_{Q_3} =  1$, we can then conclude that
\begin{enumerate}
	\item if $\sF$ is foliation by curves of degree 0 on $Q_3$, then $\nf^\vee$ is $\mu$-stable;
	\item if $\sF$ is foliation by curves of degree 1 on $Q_3$, then $\nf^\vee$ is $\mu$-semistable;
\end{enumerate}


\subsection{Generic foliations by curves on quadric threefolds} \label{sec:q3}

A generic foliation by curves $\sF$ on $Q_3$ is given by 
$$\sF ~~:~~ 0\to\mathcal{O}_{Q_3}(-r) \stackrel{\sigma}{\to} TQ_3 \to \nf \to 0$$
since $\tau_{Q_3}=0$, where $r\ge0$ is the degree of $\sF$. According to Theorem \ref{estabilidade}, the normal sheaf $\nf$ is a $\mu$-stable rank 2 reflexive sheaf on $Q_3$.

When $\sF$ has odd degree, say $r=2k+1$, then the normalization of the normal sheaf fits into the short exact sequence
\begin{equation}\label{folquadrica1}
 0\to\mathcal{O}_{Q_3}(-3-3k) \stackrel{\sigma}{\to} TQ_3(-2-k) \to \nf(-2-k) \to 0,
\end{equation}
for $k\geq 0$. Similarly, if $\sF$ has even degree, say $r=2k$, then the normalization of the normal sheaf fits into the short exact sequence 
\begin{equation}\label{folquadrica2}
 0\to\mathcal{O}_{Q_3}(-2-3k) \stackrel{\sigma}{\to} TQ_3(-2-k) \to \nf(-2-k) \to 0,
\end{equation}
where $k\geq 0$.
From now on, we will focus on foliations of odd degree; the even case can be dealt with in a similar way.

Let $\mathcal{F}(2k+1)$ denote the family of isomorphism classes of stable rank 2 reflexive sheaves on $Q_3$ given by the short exact sequences of the following form
\begin{equation}
0 \to \mathcal{O}_{Q_3}(-3-3k)\oplus \mathcal{O}_{Q_3}(-2-k) \to \Omega^1_\p4|_{Q_3}(-k) \to F \to 0.
\end{equation}
Note that if $\sF$ is a generic foliation by curves of odd degree, then $\nf(-2-k)$ belongs to the family $\mathcal{F}(2k+1)$. Indeed, we can use the isomorphism $TQ_3\simeq\Omega^{1}_{Q_3}(2)$, see \cite{KO}, to rewrite the exact sequence in display \eqref{folquadrica1} as follows
\begin{equation}\label{distquadric}
 0\to\mathcal{O}_{Q_3}(-3-3k) \stackrel{\sigma}{\to} \Omega^{1}_{Q_3}(-k) \to \nf(-2-k) \to 0.
\end{equation}
We then have the commutative diagram
\begin{equation}\label{diagram} 
\xymatrix{
 & 0 \ar[d] & 0 \ar[d] & & \\
 & \mathcal{O}_{Q_3}(-2-k) \ar@{=}[r]\ar[d] & \mathcal{O}_{Q_3}(-2-k) \ar[d]^{\eta} &  & \\
0 \ar[r] & \mathcal{O}_{Q_3}(-2-k)\oplus\mathcal{O}_{Q_3}(-3-3k) \ar[d] \ar[r]^{\hspace{1cm}\phi} & \Omega^{1}_{\p4}|_{Q_3}(-k) \ar[d] \ar[r] & \nf(-2-k) \ar@{=}[d] \ar[r] & 0 \\
0 \ar[r] & \mathcal{O}_{Q_3}(-3-3k) \ar[d] \ar[r]^{\sigma} & \Omega^{1}_{Q_3}(-k) \ar[d] \ar[r] & \nf(-2-k) \ar[r] & 0 \\
& 0 & 0 & &
}
\end{equation}
where $\phi=(\eta,\sigma)$ and $\eta$ is the morphism induced by the inclusion $Q_3\into\p4$.

On the other hand, since the quotient of an arbitrary section $\eta\in H^0(\Omega^1_\p4|_{Q_3}(-k))$ may not be isomorphic to $\Omega^1_{Q_3}$, we observe that a generic sheaf of the family $\mathcal{F}(2k+1)$ is not isomorphic to the conormal sheaf of a generic foliation by curves. In other words, the family of stable rank 2 reflexive sheaves obtained as the conormal sheaf of a generic foliation by curves is strictly contained in the family $\mathcal{F}(2k+1)$.

In order to further study the family $\mathcal{F}(2k+1)$, we need the following technical result.

\begin{lemma}\label{lem_simples}
The sheaf $\Omega^1_\p4|_{Q_3}$ is simple, i.e. $\dim\Hom(\Omega^1_\p4|_{Q_3},\Omega^1_\p4|_{Q_3})=1$.
\end{lemma}

\begin{proof}
Applying the functor $\Hom(.,\Omega^1_\p4|_{Q_3})$ to the exact sequence
\begin{equation}\label{simple}
0 \to \Omega^1_\p4(-2) \stackrel{\cdot f}{\to}  \Omega^1_\p4 \to \Omega^1_\p4|_{Q_3} \to 0,
\end{equation}
we get
\begin{equation}\label{simple'}
0 \to \Hom(\Omega^1_\p4|_{Q_3},\Omega^1_\p4|_{Q_3}) \to \Hom(\Omega^1_\p4,\Omega^1_\p4|_{Q_3}) \to \cdots
\end{equation}

Now, applying the functor $\Hom(.,\Omega^1_\p4|_{Q_3})$ to the exact sequence
\begin{center}
$0 \to \Omega^1_\p4 \to  \op4(-1)^{\oplus 5} \to \op4 \to 0,$
\end{center}
we get $\dim\Hom(\Omega^1_\p4,\Omega^1_\p4|_{Q_3})=1$, since 
$H^0(\Omega^1_\p4|_{Q_3}(1))=H^1(\Omega^1_\p4|_{Q_3}(1))=0$ and 
$h^1(\Omega^1_\p4|_{Q_3})=1$. Since, by display \eqref{simple'},
$$ 1\le \dim\Hom(\Omega^1_\p4|_{Q_3},\Omega^1_\p4|_{Q_3}) \le \dim\Hom(\Omega^1_\p4,\Omega^1_\p4|_{Q_3}) $$
we conclude that
$\dim\Hom(\Omega^1_\p4|_{Q_3},\Omega^1_\p4|_{Q_3})=1$, as desired.
\end{proof}

\begin{lemma}
Let $\phi, \phi':\mathcal{O}_{Q_3}(-2-r)\oplus \mathcal{O}_{Q_3}(-2) \to \Omega^1_\p4|_{Q_3}$ be monomorphisms such that $F:-\coker\sigma$ and $F':=\coker\sigma'$ are reflexive sheaves. $F$ and $F'$ are isomorphic if and only if there is an automorphism $\psi\in\Aut(\mathcal{O}_{Q_3}(-2-r)\oplus \mathcal{O}_{Q_3}(-2))$ with $\phi'\circ\psi=\phi.$  
\end{lemma}

\begin{proof}
If $\phi'\circ\psi=\phi$, it is easy to check that $\coker\phi$ and $\coker\phi'$ are isomorphic.

Conversely, suppose
$$\phi, \phi' : \mathcal{O}_{Q_3}(-2-r)\oplus \mathcal{O}_{Q_3}(-2) \to \Omega^1_\p4|_{Q_3}$$
are monomorphisms and
$$g:F \to F'$$
is an isomorphism between their cokernels. Applying the functor $\Hom(\Omega^1_\p4|_{Q_3},.)$ to the exact sequence
\begin{center}
$0 \to \mathcal{O}_{Q_3}(-2-r)\oplus \mathcal{O}_{Q_3}(-2) \stackrel{\phi'}{\to}  \Omega^1_\p4|_{Q_3}  \stackrel{p'}{\to} F' \to 0,$
\end{center}
we get the isomorphism
\begin{center}
$\Hom(\Omega^1_\p4|_{Q_3},\Omega^1_\p4|_{Q_3})\simeq \Hom(\Omega^1_\p4|_{Q_3},F')$
\end{center}
since
\begin{center}
$\Hom(\Omega^1_\p4|_{Q_3},\mathcal{O}_{Q_3}(-2-r)\oplus \mathcal{O}_{Q_3}(-2))\simeq H^0(T\p4|_{Q_3}(-2-r))
\oplus H^0(T\p4|_{Q_3}(-2))=0$
\end{center}
and
\begin{center}
$\Ext^1(\Omega^1_\p4|_X,\mathcal{O}_{Q_3}(-2-r)\oplus \mathcal{O}_{Q_3}(-2))\simeq H^1(T\p4|_{Q_3}(-2-r))
\oplus H^1(T\p4|_{Q_3}(-2))=0.$
\end{center}
Thus, given $\xi\in \Hom(\Omega^1_\p4|_{Q_3},F')$, there exists a unique $\lambda\in\Hom(\Omega^1_\p4|_{Q_3},\Omega^1_\p4|_{Q_3})$ such that $p'\circ \lambda=\xi$. Being $\Omega^1_\p4|_{Q_3}$ simple, by Lemma \ref{lem_simples}, it follows that $\lambda$ is a multiple of the identity morphism.

Therefore, as $g\circ p\in \Hom(\Omega^1_\p4|_{Q_3},F')$, we get the following isomorphism betweeen exact sequences 
\begin{equation*}
\xymatrix{
0 \ar[r] & \mathcal{O}_{Q_3}(-2-r)\oplus \mathcal{O}_{Q_3}(-2) \ar[d]^{\psi} \ar[r]^{\hspace{15mm}\phi} & \Omega^1_\p4|_{Q_3} \ar[d]^{\lambda} \ar[r]^{p} & F \ar[d]^{g} \ar[r] & 0 \\
0 \ar[r] & \mathcal{O}_{Q_3}(-2-r)\oplus \mathcal{O}_{Q_3}(-2) \ar[r]^{\hspace{15mm}\phi'} & \Omega^1_\p4|_{Q_3}  \ar[r]^{p'} & F'  \ar[r] & 0 \\
}
\end{equation*}
that is, there is an automorphism $\psi\in\Aut(\mathcal{O}_{Q_3}(-2-r)\oplus \mathcal{O}_{Q_3}(-2))$ such that $\sigma'\circ\psi=\sigma.$ 
\end{proof}

It follows that:
$$ \begin{array}{rcl}
 \dim \mathcal{F}(2k+1) & =   & \dim\Hom(\mathcal{O}_{Q_3}(-3-3k)\oplus \mathcal{O}_{Q_3}(-2-k),\Omega^{1}_{\p4}|_{Q_3}(-k)) \\
 &  & -\dim\Aut(\mathcal{O}_{Q_3}(-3-3k)\oplus \mathcal{O}_{Q_3}(-2-k))  \\
     & = & 8k^3+42k^2+69k+44.
\end{array} $$
On the other hand, the family $\mathcal{D}(2k+1)$ of stable rank 2 sheaves obtained as conormal sheaves of generic foliations of curves of odd degree is parametrized by an open subset of $\Hom(\mathcal{O}_{Q_3}(-3-3k),\Omega^{1}_{Q_3}(-k))$ up to a scalar factor, thus
$$\begin{array}{rcl}
 \dim \mathcal{D}(2k+1) & =   & \dim\Hom(\mathcal{O}_{Q_3}(-3-3k),\Omega^{1}_{Q_3}(-k))-1 \\
     & = & 8k^3+42k^2+69k+34.
\end{array} $$

We are finally in position to prove the main result of this section.

\begin{theorem}\label{teoquadrica}
For each $k\geq 1$, the moduli space of stable rank 2 reflexive sheaves on $Q_3$ with Chern classes
\begin{equation*}
    (c_1,c_2,c_3)=(0,(3k^2+6k+4)H^2,(8k^3+24k^2+26k+6)H^3)
\end{equation*}
\noindent possesses an irreducible component of dimension $8k^3+42k^2+69k+44$ which contains, as a closed subset, the normal sheaves of generic foliations by curves of degree $2k+1$ on $Q_3$.
\end{theorem}

\begin{proof}
Given a foliation by curves $\sF$ of odd degree, it is enough to argue that
$$ \dim\Ext^1(\nf,\nf)=\dim\mathcal{F}(2k+1)=8k^3+42k^2+69k+44. $$
for $k\geq 1$.
Since $\dim\mathcal{F}(2k+1)$ is a lower bound for $\dim\Ext^1(F,F)$ for sheaves $F\in\mathcal{F}(2k+1)$, semicontinuity allows us to conclude that $\dim\Ext^1(F,F)=\dim\mathcal{F}(2k+1)$ for a generic sheaf $F\in\mathcal{F}(2k+1)$.

Applying the functor $\Hom(.,\nf(-2-k))$ in the exact sequence in display \eqref{folquadrica1}, we get the equality
\begin{equation}\label{eq01}
 \sum_{j=0}^{3}(-1)^{j}\dim\Ext^{j}(\nf,\nf)=\chi(\Omega^{1}_{Q_3}\otimes \nf)-\chi(\nf(1+2k))   
\end{equation}

Now, we twist the exact sequences
\begin{equation}\label{seqeuler}
0 \to \Omega^{1}_{\p4}|_{Q_3} \to \mathcal{O}_{Q_3}(-1)^{\oplus 5} \to \mathcal{O}_{Q_3} \to 0
\end{equation}
and
\begin{equation}\label{seqcot}
0 \to \mathcal{O}_{Q_3}(-2) \to \Omega^{1}_{\p4}|_{Q_3} \to \Omega^{1}_{Q_3} \to 0 
\end{equation}
by $\otimes\nf$ and then taking the Euler characteristic, we get
\begin{equation}\label{eq02}
\chi(\Omega^{1}_{Q_3}\otimes \nf)-\chi(\nf(1+2k))=-36k^2-72k-44.
\end{equation}

Here and from \eqref{eq01}, it follows that
$$ \dim\Ext^1(\nf,\nf)-\dim\Ext^2(\nf,\nf)=36k^2+72k+45, $$
since $\dim\Hom(\nf,\nf)=1$ and $\dim\Ext^3(\nf,\nf)=0.$ Therefore, we must now show that
$$\dim\Ext^2(\nf,\nf)=\dim\mathcal{F}(2k+1)-36k^2-72k-45=8k^3+6k^2-3k-1. $$

Applying the functor $\Hom(.,\nf(-2-k))$ to the exact sequence in display \eqref{folquadrica1}, we obtain the isomorphism
$$
\Ext^2(\nf,\nf)\simeq \Ext^2(TQ_3,\nf)\simeq H^2(TQ_3\otimes \nf(-2)),
$$
since $h^1(\nf(1+2k))=h^2(\nf(1+2k))=0$ by Corollary \eqref{cohquadrica}.

Now, we twist the dual exact sequence in display \ref{seqcot} by $\otimes\nf(-2)$ and pass to cohomology, obtaining 
$$  h^2(TQ_3\otimes \nf(-2))= h^2(T\p4|_{Q_3}\otimes \nf(-2))-h^2(\nf), $$
since $H^1(\nf)=H^3(TQ_3\otimes\nf(-2))=0$. 

To compute $h^2(T\p4|_{Q_3}\otimes \nf(-2))$, we twist the exact sequences of the second line and last column in the diagram \eqref{diagram} 
by $\otimes T\p4|_{Q_3}(k)$ and $\otimes T\p4|_{Q_3}(-2)$, respectively, and
pass to cohomology, we have
$$
h^2(T\p4|_{Q_3}\otimes \nf(-2))=h^3(T\p4|_{Q_3}(-3-2k))=\frac{1}{3}(2k-1)(2k+1)(8k+3),
$$
since $h^2(T\mathbb{P}^4|_{Q_3}(-2))=h^3(T\mathbb{P}^4|_{Q_3}(-2))=h^2(T\mathbb{P}^4|_{Q_3} \otimes \Omega^{1}_{\mathbb{P}^4}|_{Q_3})=h^3(T\mathbb{P}^4|_{Q_3} \otimes \Omega^{1}_{\mathbb{P}^4}|_{Q_3})=0.$

So, for $k\geq 1$, we have
$$ \dim\Ext^2(\nf,\nf) =  h^3(T\p4|_{Q_3}(-3-2k))-h^2(\nf)=  8k^3+6k^2-3k-1, $$
since $h^2(\nf)=h^0(\mathcal{O}_{Q_3}(-2+2k))$ by Corollary \ref{cohquadrica}.
\end{proof}

\begin{remark}
When $k=0$, we can still conclude that the sheaves $\nf$ given by the generic foliations by curves of degree 1 on $Q_3$ are smooth points of the moduli space of stable rank 2 reflexive sheaves with
Chern classes $(c_1,c_2,c_3)=(0,4H^2,6H^3)$ within an irreducible component of dimension 45, since $\Ext^2(\nf,\nf)=0.$
However, these sheaves only form a family of dimension 44 within this irreducible component.
\end{remark}

For the generic foliations by curves of degree $r=2k$ on $Q_3$ we have the following theorem, whose proof is analogous to proof of the Theorem \ref{teoquadrica}. We observe that the family of conormal sheaves of foliations by curves of even degree has the following dimension:
$$\begin{array}{rcl}
\dim \mathcal{D}(2k) & = & \dim\Hom(\mathcal{O}_{Q_3}(-2-3k),\Omega^{1}_{Q_3}(-k))-1 \\
     & = & 8k^3+30k^2+33k+9.
\end{array} $$

\begin{theorem}
For each $k\geq 0$, the moduli space of stable rank 2 reflexive sheaves on $Q_3$ with Chern classes
$$(c_1,c_2,c_3)=(-H,(3k^2+3k+2)H^2,(8k^3+12k^2+8k-2)H^3)$$
possesses an irreducible component of dimension $18$ for $k=0$, and $8k^3+30k^2+33k+19$, for $k\geq 1$ which contains, as a closed subset, the normal sheaves of generic foliations by curves of degree $2k$ on $Q_3$.
\end{theorem}


\subsection{Spinor foliations}
Let us recall the definition and some properties of the \emph{spinor bundle} on $Q_3$. In particular, we revisit Ottaviani’s geometrical construction \cite{Ott1,Ott2}.

Let $G(2, 4)$ denote the Grassmannian of all 2-dimensional linear subspaces of $\K^4$. By using the geometry of the variety of all 1-dimensional linear subspaces of $Q_3$ it is possible to construct a morphism $s : Q_3 \to G(2, 4).$ Let $U$ be the universal bundle of the Grassmannian. 

\begin{definition}
	The pull-back bundle $S:=s^{\ast}U$ is called the \textit{spinor bundle} on $Q_3.$ 
\end{definition}

It is easy to see that $S$ is a rank 2 vector bundle; Ottaviani also shows that $S$ is $\mu$-stable ~\cite[Theorem 2.1]{Ott1}, that $S(1)$ is globally generated  and it satisfies $S^{\vee}=S(1)$ \cite[Theorem 2.8(i)]{Ott1}. In addition, $S$ is the unique stable rank 2 bundle on $Q_3$ with $c_1(S) = -H$ and $c_2(S) = L,$ cf. \cite{AS}, and fits into the short exact sequence
\begin{equation} \label{spinor-line}
	0\to \mO_{Q_3}(-1) \to S \to \sI_L \to 0
\end{equation}
where $L$ is a line in $Q_3$.


Since the dual of the spinor bundle $S^{\vee}$ is globally generated, $S^{\vee}(t)$ is globally generated, for all $t \geq 0,$ and thus so is $S^{\vee} \otimes \Omega^1_{Q_3}(2+t)$. We can then apply Ottaviani's Bertini-type theorem \cite[Teorema 2.8]{O}, to show that there are foliations by curves of the form
\begin{equation} \label{spinor fol}
	\sF_t: 0 \to S(-2-t) \to \Omega^1_{Q_3} \to \sI_{C}(2+2t) \to 0;
\end{equation} 
for each $t\ge0$; note that $\deg(\sF)=2+2t$.

Furthermore, we observe that the $\mu$-stability of $\Omega^1_{Q_3}$ implies that there are no injective morphisms $S(-2-t) \into \Omega^1_{Q_3}$ when $t\le-2$, since
$$ \mu(S(-2-t))=-(2t+5)/2>-1=\mu(\Omega^1_{Q_3}). $$

Finally, we consider the case $t=-1$ observing that $\Hom(S(-1),\Omega^1_{Q_3})\simeq H^0(S\otimes\Omega^1_{Q_3}(2))$. First, twist the exact sequence
\begin{equation}\label{normal-bundle}
	0 \to \mO_{Q_3}(-2) \to \Omega \to \Omega^1_{Q_3} \to 0, 
\end{equation}
where $\Omega:=\Omega^1_{\p4}|_{Q_3}$, by $S(2)$ to conclude that $H^0(S\otimes\Omega^1_{Q_3}(2))\simeq H^0(S\otimes\Omega(2))$, since $h^0(S)=h^1(S)=0$. Next, noting that $H^0(\mO_{Q_3}(1))\simeq H^0(\op4(1))$, we use the Euler exact sequence for $\Omega^1_{\p4}$ restricted to $Q_3$, namely
\begin{equation}\label{seq euler}
	0 \to \Omega \to H^0(\mO_{Q_3}(1))\otimes \mO_{Q_3}(-1) \to \mO_{Q_3} \to 0; 
\end{equation}
twisting it by $S(2)$, we obtain the multiplication map $\mu:H^0(\mO_{Q_3}(1))\otimes H^0(S(1)) \to H^0(S(2))$. Lemma \ref{tec-lem} below guarantees that $\mu$ is surjective. Since $\ker\mu\simeq H^0(S\otimes\Omega(2))$, it follows that
\begin{equation}\label{dim hom s}
	h^0(S\otimes\Omega(2)) = 5\cdot h^0(S(1)) - h^0(S(2)) = 4.    
\end{equation}

\begin{lemma}\label{tec-lem}
	Let $F$ be a globally generated coherent sheaf on a non singular threefold hypersurface $X$, and let 
	$$ \varepsilon:H^0(F)\otimes\mathcal{O}_X\twoheadrightarrow F \text{ and }
	\beta:H^0(\mathcal{O}_X(1))\otimes\mathcal{O}_X \twoheadrightarrow \mathcal{O}_X(1) $$
	be the evaluation morphisms; set $G:=\ker\varepsilon$. If $H^1(G(1))=0$, then the multiplication map
	$$ \mu : H^0(F)\otimes H^0(\mathcal{O}_X(1)) \longrightarrow  H^0(F(1)) $$
	is surjective
\end{lemma}

This result applies nicely when $X$ is a quadric threefold and $F=S(1)$ is the twisted spinor bundle, $k\ge1$; in this case, $G=S\oplus\mO_{Q_3}$ (compare with \cite[Theorem 2. (i)]{Ott1}), we know that $H^1(G(1))=0$.

\begin{proof}
	Twisting the exact sequence
	$$ 0\to G\to H^0(F)\otimes\mathcal{O}_X \to F \to 0 $$
	by $\mathcal{O}_X(1)$ and passing to cohomology, we obtain 
	$$ H^0(F)\otimes H^0(\mathcal{O}_X(1)) \stackrel{\mu}{\longrightarrow} H^0(F(1)); $$
	note that this is surjective precisely when $H^1(G(1))=0$.
\end{proof}

To conclude the discussion of the case $t=-1$, we prove:

\begin{lemma}\label{spin fol -1}
	Every non trivial morphism $\varphi:S(-1)\to \Omega^1_{Q_3}$ is injective and has torsion free cokernel.
\end{lemma}
\begin{proof}
	If a non trivial morphism $\varphi:S(-1)\to \Omega^1_{Q_3}$ is not injective, then $\ker\varphi\simeq\mO_{Q_3}(-k)$ for some $k\ge2$, since $\ker\varphi$ must be a rank 1 reflexive sheaf. Furthermore, $\im\varphi\simeq \sI_X(k-3)$ for some curve $X$; since this is a subsheaf of $\Omega^1_{Q_3}$, we must have $k-3\le-2$, leading to a contradiction.
	
	Given a monomorphism $\varphi:S(-1)\to \Omega^1_{Q_3}$, let $P:=\coker\varphi$; note that $c_1(P)=0$. If $P$ is not torsion free, let $T\into P$ be its maximal torsion sheaf, so that $P/T\simeq \sI_Y(k)$; we note that $c_1(P/T)\ge0$ since this is a quotient of $\Omega^1_{Q_3}$. Since $c_1(T)=-c_1(P/T)\le0$, it follows that $c_1(T)=-c_1(P/T)=0$, so $\dim T\le1$. But we get, from the exact sequence $0\to T\to P\to P/T\to0$, that $\inext^q(P,\mO_{Q_3})\ne0$ either for $q=2$ (if $\dim T=1$) or $q=3$ (if $\dim T=0$). But this contradicts the sequence $0\to S(-1)\stackrel{\varphi}{\to} \Omega^1_{Q_3}\to P\to0$, which implies that $\inext^q(P,\mO_{Q_3})=0$ for $q>1$.
\end{proof}

The results above lead us to the following definition.

\begin{definition}
	A \emph{spinor foliation} on $Q_3$ is a foliation by curves whose conormal sheaf is isomorphic to the spinor bundle up to twist, that is, $\nf^\vee \simeq S(-2-t)$ for some $t \geq -1.$
\end{definition}

The singular scheme $C$ of the spinor foliation $\sF_t$ is a connected curve (by Corollary \ref{q-conn}) of degree $\deg C= 6t^2+18t+15$ and genus $g = 10t^3+36t^2+43t+16$ (see Corollary \ref{deg+chi}). In addition, $C$ is an arithmetically Buchsbaum curve \cite[Theorem 4.1]{AMS}; Ottaviani's Bertini-type theorem implies that $C$ is smooth for generic choice of monomorphism $S(-2-t) \into \Omega^1_{Q_3}$.


\subsection{Foliations of odd degree}\label{OS bundles}

Let $E$ be a $\mu$-semistable rank 2 vector bundle on $Q_3$ with Chern classes $c_1(E)=0$ (so that $E^{\ast}\simeq E$) and $c_2=2L$. One can show that

\begin{enumerate}
	\item either $E$ is stable and is given by and extension of $\sI_X(1)$ by $\mO_{Q_3}(-1)$, where $X$ is the union of two disjoint conics;
	\item or $E$ is strictly $\mu$-semistable and is given by and extension of $\sI_Y$ by $\mO_{Q_3}$, where $Y$ is a double line of genus $-2$.
\end{enumerate}


Indeed, the first claim was proved by Ottaviani and Szurek in \cite[Section 2]{Ott-Suz}. When $E$ is strictly $\mu$-semistable, we consider a non trivial section $\sigma\in H^0(E)$, and let $Y:=(\sigma)_0$ be its zero locus. This gives us the short exact sequence
$$ 0\to \mO_{Q_3} \stackrel{\sigma}{\to} E \to \sI_Y \to 0; $$
The numerical invariants of $Y$ can easily be computed from the Chern Classes of $E$:
$$ \deg(Y)=c_2(E)=2 ~~~{\rm and }~~~ 2p_a(Y)-2=-3c_2(E)=-6 ~~\Rightarrow~~ p_a(Y)=-2,$$
as desired. One can check that 
$$ \Ext^1(\sI_Y,\mO_{Q_3}) \simeq H^2(\sI_Y(-3))^{\vee} \simeq H^1(\mO_{L}(-2))^{\vee}\simeq \C, $$
where $L:=Y_{\rm red}$ is the line supporting $Y$; morever, note that $\omega_Y\simeq\mO_{Y}\otimes\omega_{Q_3}$, implying that the unique nontrivial extension of $\sI_Y$ by $\mO_{Q_3}$ is indeed locally free.


Let us now focus on the first case assuming that $E$ is $\mu$-stable. Ottaviani and Szurek show that $E(1)$ is globally generated \cite[Proposition 1.11]{Ott-Suz}.

Since $E^{\vee}(1)$ is globally generated, $E^{\vee}(t)$ is globally generated, for all $t\geq 1$, and hence $E^{\vee} \otimes \Omega^1_{Q_3}(2+t)$ is also globally generated in this range. By  Ottaviani's Bertini-type thorem \cite[Teorema 2.8]{O}, there is a foliation by curves of the form
\begin{equation} \label{odd fol}
	\sF_t: 0 \to E(-2-t) \to \Omega^1_{Q_3} \to \sI_{C}(2t+1) \to 0, ~~ t\ge1;
\end{equation} 
note that $\deg(\sF)=2t+1$. By Corollary \ref{deg+chi}, we have that the singular scheme $C$ of a foliation of the form \eqref{odd fol} is a (generically smooth) curve of degree $\deg C=22t^2-48t+24$ and genus $g=58t^3-219t^2+262t-97$. 

Furthermore, we observe that a sequence as in display \eqref{odd fol} does not exist for $t\le-1$: since $\Omega^1_{Q_3}$ is $\mu$-stable, we must have that 
$$ \mu(E(-2-t))=-2-t < -1=\mu(\Omega^1_{Q_3}). $$

Finally, let us consider the critical case $t=0$. First, twist the exact sequence in display \eqref{normal-bundle} by $E(2)$ and then taking the Euler characteristic, we get 
$$\chi(E\otimes \Omega^1_{Q_3}(2))=\chi(E\otimes \Omega(2))+1,$$
since $\chi(E)=-1.$ Now, twist the exact sequence in display \eqref{seq euler} by $E(2)$ and then taking the Euler characteristic,
we get $\chi(E\otimes \Omega(2))=8$, since $\chi(E(1))=5$ and $\chi(E(2))=17$. 

\noindent Thus  $\chi(E\otimes \Omega^1_{Q_3}(2))=9.$
Now, we will show that $h^0(E\otimes \Omega^1_{Q_3}(2))\geq 9$. For this, it is sufficient to argue that  $h^2(E\otimes \Omega^1_{Q_3}(2))=0.$
Initially note that, by the exact sequence in display \eqref{normal-bundle}, $h^2(E\otimes \Omega^1_{Q_3}(2))=h^2(E\otimes \Omega(2)),$ since $h^2(E)=h^3(E)=0.$ Next, twist the exact sequence in display \eqref{seq euler} by $E(2)$ and then taking the long exact sequence in cohomology, we have
$$\cdots \to H^1(E(2)) \to H^2(E\otimes \Omega(2)) \to H^2(E(1)^{\oplus 5})\to\cdots$$ 

As $h^1(E(2))=h^2(E(1))=0$, we get $h^2(E\otimes \Omega(2))=h^2(E\otimes \Omega^1_{Q_3}(2))=0$ and hence
$$h^0(E\otimes \Omega^1_{Q_3}(2))\geq 9,$$
since $\chi(E\otimes \Omega^1_{Q_3}(2))=9.$
Therefore, there is non trivial morphism $\phi: E(-2)\to \Omega^1_{Q_3}$.

Similarly, when $E$ is strictly $\mu$-semistable, we get $\chi(E\otimes \Omega^1_{Q_3}(2))=9$ and  $$h^2(E\otimes \Omega^1_{Q_3}(2))=h^2(E\otimes \Omega(2)).$$ To show that $h^0(E\otimes \Omega^1_{Q_3}(2))\neq 0$, we will argue that $h^2(E\otimes \Omega(2))\leq 8$. Indeed, twist the exact sequence in display \eqref{seq euler} by $E(2)$ and then taking the long exact sequence in cohomology, we get the epimorphism
$$\cdots \to H^1(E(2)) \to H^2(E\otimes \Omega(2)) \to 0,$$
since $h^2(E(1))=0.$ Now, from the exact sequence 
$$ 0\to \mO_{Q_3} \stackrel{\sigma}{\to} E \to \sI_Y \to 0,$$ we get $h^1(E(2))=h^1(\sI_Y(2))\leq 4$, since $Y$ is a double line. Therefore,
$h^2(E\otimes \Omega^1_{Q_3}(2))\leq 4$ and hence $h^0(E\otimes \Omega^1_{Q_3}(2))\geq 5.$

In fact, every non trivial morphism $\phi: E(-2)\to\Omega^1_{Q_3}$ is a monomorphism: if $\phi$ is not injective, then $\ker\phi\simeq\mO_{Q_3}(-k)$ for some $k\geq 3,$ since $\ker\phi$ must be a rank 1 reflexive sheaf. Thus, $\im\phi\simeq \sI_Z(k-4)$ for some curve
$Z\subset Q_3$, since this a rank 1 subsheaf of $\Omega^1_{Q_3}.$ The stability of $\Omega^1_{Q_3}$ implies that $k-4\leq -2$, leading to a contradiction. However, it is not clear to the authors whether there exists $\phi\in\Hom(E(-2),\Omega^1_{Q_3})$ such that $\coker\phi$ is torsion free.


\subsection{Local complete intersection foliations of degree 0}

Let us now consider consider local complete intersection foliations by curves of degree 0 on a smooth quadric threefold $Q_3$, given by exact sequences of the following form
\begin{equation} \label{deg 0 fol}
	\sF: 0 \to \nf^\vee \to \Omega^1_{Q_3} \to \sI_{C} \to 0.    
\end{equation}

\begin{proposition}
	If $\sF$ is a local complete intersection foliations of degree 0 a smooth quadric hypersurface $Q_3\subset\p4$, then $\nf^\vee=S(-1)$, where $S$ is the spinor bundle, and $\sing(\sF)$ is the disjoint union of a line and a conic.
\end{proposition}
\begin{proof}
	Theorem \ref{invariantes} yields $c_1(\nf^\vee) = -3H$ and $c_2(\nf^\vee) = 4H^2- [C].$ Since $\nf^\vee$ is $\mu$-stable, we can use the Bogomolov inequality to obtain the following bound on the degree of the singular curve $C$:
	$$ \int_X c_2(\nf^\vee)H \ge \dfrac{1}{4}\int_X c_1(\nf^\vee)^2 H = \dfrac{9}{2} ~~ \Rightarrow~~
	\deg(C) \le 7/2. $$
	
	If $\sF$ is not generic, we can conclude that $1\le\deg(C)\le3$. In addition, Corollary \ref{deg+chi} yields $\chi(\cO_C) = 2$. Let us analyze each possibility.
	
	If $\deg(C) = 1,$ then $C$ must be a line, contradicting the equality $\chi(\cO_C) = 2$.
	
	If $\deg(C) = 2,$ so that $C$ is $2$ disjoint lines or a double line, and hence $C$ is a 1-Buchsbaum curve.
	
	The same happens when $\deg(C) = 3$: a curve of degree 3 and genus $-1$ must be a conic plus a line which is again a 1-Buchsbaum curve.
	
	It follows from \cite[Theorem 4.4]{AMS}, that $\nf^\vee$ is arithmetically Cohen--Macaulay, so $\nf^\vee$ is a direct sum of line bundles and twisted spinor bundles by \cite[Theorem 2.7]{Bu}. Since $\rk(\nf^\vee)=2$, we are left with two possibilities: either $\nf^\vee=\mO_{Q_3}(a)\oplus\mO_{Q_3}(b)$ with $a+b=-3$, or $\nf^\vee=S(-1)$.
	
	The first possibility is easy to rule out, since we must have $a,b\le-2$, while the second case does ineed occur, by inequality in display \eqref{dim hom s} and Lemma \ref{spin fol -1}.
\end{proof}

\bigskip

\subsection{Local complete intersection foliations of degree 1}

Let us now consider consider local complete intersection foliations by curves of degree 1 on a smooth quadric threefold $Q_3$, given by exact sequences of the following form
\begin{equation} \label{deg 1 fol}
	\sF: 0 \to \nf^\vee \to \Omega^1_{Q_3} \to \sI_{C}(1) \to 0.    
\end{equation}
Furthermore, Theorem \ref{invariantes} yields $c_2(\nf^\vee) = -4H$ and $c_2(\nf^\vee) = 8H^2 - [C].$

\begin{proposition}
	Let $\sF$ is a local complete intersection foliations of degree 1 with reduced singular scheme on a smooth quadric hypersurface $Q_3\subset\p4$, and set $E:=\nf^\vee(2)$ and $C:=\sing(\sF)$. Then
	
	then 
	\begin{enumerate}
		\item $E$ is the $\mu$-stable bundle with Chern classes $c_1(E)=0$ and $c_2(E)=2L$ and $C$ is a rational curve of degree 6;
		\item $E$ is the $\mu$-semistable bundle with Chern classes $c_1(E)=0$ and $c_2(E)=2L$ and $C$ is a curve of degree 6 given by the union of a rational and an elliptic curves.
		\item $E=\mathcal{O}_Q^{\oplus 2}$ and $C$ is a connected curve of degree 8 and genus 3.
	\end{enumerate}
\end{proposition}

We remark that the bundles featured in items (1) and (2) above are precisely the one described in detail in Section \ref{OS bundles}.

\begin{proof}
	By Bogomolov's inequality, $\Delta(\nf^\vee)H = 16H^3-4[C]H \geq 0.$ Thus, $\deg C \leq 8.$  On the other hand, since $\nf^\vee$ is locally free, Corollary \ref{deg+chi} yields $\chi(\cO_C) = 10-\frac{3}{2} \deg (C)$. If $\deg(C)=2$, then $\chi(\cO_C)=7$, thus $C$ cannot be reduced. It follows that $\deg (C)=4,6,8$ since the $\chi(\cO_C)$ is a integer.
	
	
	If $\deg(C)=4$, then $\chi(\cO_C)=4$ and the reducedness assumption implies that $C$ consists of four skew lines. The conormal sheaf has Chern classes $c_1(\nf^\vee) =-4H$ and $c_2(\nf^\vee) =12L$. So, if $E:=\nf^\vee(2)$, then $E$ is a $\mu$-semistable rank 2 locally sheaf with $c_1(E)=0$ and $c_2(E)=4L$. It follows that $h^0(E(1))>0$, since $h^0(E)\ne0$ if $E$ is strictly $\mu$-semistable and by \cite[Proposition 3.1]{Ott-Suz} if $E$ is $\mu$-stable, thus starting from a non trivial section $\sigma\in H^0(E(1))$ we can build the following commutative diagram:
	$$ \xymatrix{
		& 0 \ar[d] & 0\ar[d] & & \\
		& \mO_{Q_3}(-3) \ar[d]\ar@{=}[r] & \mO_{Q_3}(-3) \ar[d] & & \\
		0\ar[r] & E(-2) \ar[d]\ar[r] & \Omega^1_{Q_3} \ar[d]\ar[r] & \mathcal{I}_C(1)\ar[r]\ar@{=}[d] & 0 \\
		0\ar[r] & \mathcal{I}_X(-1) \ar[d]\ar[r] & G \ar[d]\ar[r] & \mathcal{I}_C(1)\ar[r] & 0 \\
		& 0  & 0 & &
	} $$
	where $X:=(\sigma)_0$ is the vanishing locus of $\sigma$. Dualizing the bottom line, we obtain the exact sequence
	$$ 0\to \mO_{Q_3}(-1) \to G^{\vee} \to \mO_{Q_3}(1) 
	\stackrel{\eta}{\to} \omega_C(2) \to \cdots $$
	inducing a section $\eta\in H^0(\omega_C(1))$. Since $C$ is the union of 4 skew lines, we have that $h^0(\omega_C(1))=0$ thus $G^{\vee}\simeq\mO_{Q_3}(-1)\oplus\mO_{Q_3}(1)$. However, dualizing the middle column would induce a monomorphism $G^{\vee}\into TQ_3$, which can not exist since $h^0(TQ_3(-1))=0$. 
	
	If $\deg(C)=6$, then $\chi(\cO_C)=1$ and the conormal sheaf has Chern classes $c_1(\nf^\vee) =-4H$ and $c_2(\nf^\vee) =10L$; set $E:=\nf^\vee(2)$, so that $E$ is a rank 2 locally sheaf with $c_1(E)=0$ and $c_2(E)=2L$. Dualizing the exact sequence $ 0 \to E(-2) \to \Omega^1_{Q_3} \to \sI_{C}(1) \to 0$, we obtain
	$$ 0\to\mO_{Q_3}(-1)\to TQ_3 \to E(2) \to \omega_C(2)\to 0, $$
	and one can check that $h^0(E)=h^0(\omega_C)$. Since $E$ is $\mu$-semistable, we have two possibilities: either $E$ is $\mu$-stable and $h^0(E)=0$, or $E$ is strictly $\mu$-semistable and $h^0(E)=1$. Note that $h^0(\omega_C)=h^1(\sI_C)$, so $C$ is connected when $E$ is $\mu$-stable, and has 2 connected components when $E$ is strictly $\mu$-semistable. In this last case, say $C=C_1\sqcup C_2$ so that $\chi(\cO_{C_1})+\chi(\cO_{C_2})=\chi(\cO_C)=1$, hence $\chi(\cO_{C_1})=0$ and $\chi(\cO_{C_2})=1$ (or the other way around), and we get that $C_1$ is rational and $C_2$ is elliptic.
	
	Finally, if $\deg(C)=8$, then $\chi(\cO_C)=-2$ and the conormal sheaf has Chern classes $c_1(\nf^\vee) =-4H$ and $c_2(\nf^\vee)=4L$; set $E:=\nf^\vee(2)$, so that $E$ is a $\mu$-semistable rank 2 locally sheaf with $c_1(E)=0$ and $c_2(E)=0$. Therefore, only possibility is $E=\mathcal{O}_Q^{\oplus 2}$ hence $\nf^\vee=\mathcal{O}_Q(-2)^{\oplus 2}$. Furthermore, one can use the exact sequence \eqref{deg 1 fol} to check that $h^1(\sI_C)=0$, so $C$ is connected.
\end{proof}

\section*{Acknowledgments}

The authors thank Alan Muniz, Maurício Corrêa and Daniele Faenzi for useful discussions and suggestions.

%
%
	\subsection*{Financial disclosure}
MJ is supported by the CNPQ grant number 302889/2018-3 and the FAPESP Thematic Project 2018/21391-1. DS is supported by a PhD grant from CNPQ, and some of the results presented are part of his thesis. The authors also acknowledge the financial support from Coordenação de Aperfeiçoamento de Pessoal de Nível Superior - Brasil (CAPES) - Finance Code 001.
%
%


\begin{thebibliography}{CMP06}

\bibitem{CM}
C. Araujo, M. Corr\^ea,
{\it On degeneracy scheme of maps of vector bundles and application to holomorphic foliations,}
Math.  Z. {\bf 276} (2013), 505--515.


\bibitem{AS}
E. Arrondo, I. Sols,
{\it Classification of smooth congruence of low degree, }
J. Reine Angew. Math., {\bf 393} (1989), 199--219.

\bibitem{CJ}
M. Corr\^ea, M. Jardim,
{\it Bounds for sectional genera of varieties invariant under Pfaff fields,}
Illinois J. Math. {\bf 56} (2012), 343--352.

\bibitem{Bu}
E. Ballico, F. Malaspina, P. Valabrega, M. Valenzano,
{\it On Buchsbaum bundles on quadric hypersurfaces,}
Cent. Eur. J. Math. {\bf 10} (2012), 1361--1379.

\bibitem{CCJ1}
O. Calvo-Andrade, M. Corr\^ea, M. Jardim,
{\it Codimension one holomorphic distributions on projective threespaces,}
Internat. Math. Res. Not. {\bf 23} (2020), 9011--9074.

\bibitem{CCJ2}
O. Calvo-Andrade, M. Corr\^ea, M. Jardim,
{\it Stable rank 2 reflexive sheaves and distributions with isolated singularities on threefolds,}
To appear in Annals of the Brazilian Academy of Sciences (2021).
Preprint arXiv:1812.11355.

\bibitem{AMS}
A. Cavalcante, M. Corrêa, S. Marchesi,
{\it On holomorphic distributions on Fano threefolds,}
J. Pure Appl. Algebra \textbf{224} (2020), 106272.


\bibitem{CJM}
M. Corr\^ea, M. Jardim, S. Marchesi,
{\it Classification of foliations by curves of low degree on the three-dimensional projective space,}
Preprint arXiv: 1909.06590.


\bibitem{RH2}  
R. Hartshorne,
{\it Stable reflexive sheaves,}
Math. Ann. \textbf{254} (1980), 121--176.

\bibitem{KO}
S. Kobayashi, T. Ochiai,
Holomorphic structures modeled after hyperquadrics.
Tohoku Math. J. \textbf{34} (1982), 587--629.


\bibitem{NC}
G. Nonato Costa,
{\it Holomorphic foliations by curves on $\p3$ with non-isolated singularities,}
Ann. Fac. Sci. Toulouse Math. \textbf{15} (2006), 297--321.

\bibitem{Ott1}
G. Ottaviani,
{\it Spinor bundles on quadrics.,}
Trans. Am. Math. Soc., {\bf 307} (1988), 301--316.

\bibitem{Ott2}
G. Ottaviani, 
{\it Some extensions of Horrocks criterion to vector bundles on Grassmannians and Quadrics,}
Ann. Mat. Pura Appl. \textbf{CLV} (1989), 317--341.

\bibitem{O}
G. Ottaviani, 
{\it Variet\'a proiettive de codimension picola,}
Quaderni INDAM, Aracne, Roma, (1995).

\bibitem{Ott-Suz} G. Ottaviani, M. Szurek,
{\it On Moduli of Stable 2-Bundles with Small Chern Classes on $Q_3$,}
Ann. Mat. Pura Appl. {\bf CLXVII} (1994), 191--241



\end{thebibliography}

\end{document}